\documentclass[12pt]{amsart}
\usepackage{amsmath,amsthm,amsfonts,amssymb,mathrsfs}
\date{\today}

\usepackage[cp1251]{inputenc}         % Кодова сторінка Windows1251

\usepackage[ukrainian]{babel} % Підтримка кирилиці
\usepackage{amsmath,amsthm,amsfonts,amssymb,mathrsfs}

\usepackage{color}

\usepackage{hyperref}

\newtheorem{theorem}{Теорема}%[section]

\newtheorem{proposition}[theorem]{Твердження}
\newtheorem{corollary}[theorem]{Наслiдок}
\newtheorem{lemma}[theorem]{Лема}
\theoremstyle{definition}
\begin{document}

\title[Напiвгрупа зiркових часткових гомеоморфiзмiв ...]{Напiвгрупа зiркових часткових гомеоморфiзмiв скiнченновимiрного евклiдового простору}

\author[Олег~Гутік, Катерина Мельник]{Олег~Гутік, Катерина Мельник}
\address{Механіко-математичний факультет, Львівський національний університет ім. Івана Франка, Університецька 1, Львів, 79000, Україна}
\email{oleg.gutik@lnu.edu.ua,
ovgutik@yahoo.com, chepil.kate@gmail.com}



\begin{abstract}
Введено поняття зiркового часткового гомеоморфiзму скiнченновимiрного евклiдового простору $\mathbb{R}^n$ і досліджується структура напівгрупи $\mathbf{PStH}_{\mathbb{R}^n}$ зiркових часткових гомеоморфiзмiв простору $\mathbb{R}^n$. Описано струк\-ту\-ру ідемпотентів напівгрупи $\mathbf{PStH}_{\mathbb{R}^n}$ і відношення Ґріна на $\mathbf{PStH}_{\mathbb{R}^n}$. Зокрема доведено, що $\mathbf{PStH}_{\mathbb{R}^n}$ --- біпроста інверсна напівгрупа, а також, що кожна неодинична конгруенція на $\mathbf{PStH}_{\mathbb{R}^n}$ є груповою.

\bigskip

Oleg Gutik, Kateryna Melnyk. \emph{The semigroup of star partial homeomorphisms of a finite deminsional Euclidean space.}

In the paper the notion of a star partial homeomorphism of a finite dimensional Euclidean space $\mathbb{R}^n$ is introduced. We describe the structure of the semigroup $\mathbf{PStH}_{\mathbb{R}^n}$ of star partial homeomorphisms of the space $\mathbb{R}^n.$ The structure of the band of $\mathbf{PStH}_{\mathbb{R}^n}$ and Green's relations on $\mathbf{PStH}_{\mathbb{R}^n}$ are described. We show that $\mathbf{PStH}_{\mathbb{R}^n}$ is a bisimple inverse semigroup and every non-unit congruence on $\mathbf{PStH}_{\mathbb{R}^n}$ is a group congruence.
\end{abstract}

\maketitle

%\section{Термінологія та означення}

Ми користуватимемося термінологією з \cite{CliffordPreston1961-1967, Engelking1989, Lawson1998, Petrich1984}.

Надалі будемо вважати, що на $n$-вимірному евклідовому просторі $\mathbb{R}^n$ визначена звичайна (евклідова) топологія.

Якщо визначене часткове відображення $\alpha\colon X\rightharpoonup Y$ з множини $X$ у множину $Y$, то через $\operatorname{dom}\alpha$ i $\operatorname{ran}\alpha$ будемо позначати його \emph{область визначення} та \emph{область значень}, відповідно, а через $(x)\alpha$ і $(A)\alpha$ --- образи елемента $x\in\operatorname{dom}\alpha$ та підмножини $A\subseteq\operatorname{dom}\alpha$ при частковому відображенні $\alpha$, відповідно.

Часткове відображення (перетворення) $\alpha\colon X\rightharpoonup X$ топологічного простору $X$ називається \emph{частковим гомеоморфізмом} простору $\mathbb{R}$, якщо його звуження $\alpha|_{\operatorname{dom}\alpha} \colon \operatorname{dom}\alpha\rightarrow \operatorname{ran}\alpha$ є гомеоморфізмом.

Якщо $S$~--- напівгрупа, то її підмножина ідемпотентів позначається через $E(S)$.  Напівгрупа $S$ називається \emph{інверсною}, якщо для довільного її елемента $x$ існує єдиний елемент $x^{-1}\in S$ такий, що $xx^{-1}x=x$ та $x^{-1}xx^{-1}=x^{-1}$. В інверсній напівгрупі $S$ вище означений елемент $x^{-1}$ називається \emph{інверсним до} $x$. \emph{В'язка}~--- це напівгрупа ідемпотентів, а \emph{напівґратка}~--- це комутативна в'язка. %Надалі через $(\mathscr{P}_{\!\infty}(\mathbb{R}),\cup)$ позначатимемо \emph{вільну напівгратку} з одиницею над множиною дійсних чисел, тобто множину усіх скінченних (включно з порожньою) підмножин множини $\mathbb{R}$ з операцією об'єднання.

Відношення еквівалентності $\mathfrak{K}$ на напівгрупі $S$ називається \emph{конгруенцією}, якщо для елементів $a$ i $b$ напівгрупи $S$ з того, що виконується умова $(a,b)\in\mathfrak{K}$ випливає, що $(ca,cb), (ad,bd) \in\mathfrak{K}$, для всіх $c,d\in S$. Відношення $(a,b)\in\mathfrak{K}$ ми також будемо записувати $a\mathfrak{K}b$, і в цьому випадку будемо говорити, що \emph{елементи $a$ i $b$ є $\mathfrak{K}$-еквівалентними}. На кожній напівгрупі $S$ існують такі конгруенції: \emph{універсальна} $\mathfrak{U}_S=S\times S$ та \emph{одинична} (\emph{діагональ}) $\Delta_S=\{(s,s)\colon s\in S\}$. Такі конгруенції називаються \emph{тривіальними}. Кожен двобічний ідеал $I$ напівгрупи $S$ породжує на ній конгруенцію Ріса: $\mathfrak{K}_I=(I\times I)\cup\Delta_S$. Конгруенція $\mathfrak{K}$ на напівгрупі $S$ називається \emph{груповою}, якщо фактор-напівгрупа $S/\mathfrak{K}$ є групою.

Якщо $S$~--- напівгрупа, то на $E(S)$ визначено частковий порядок
\begin{equation*}
    e\leqslant f \quad \hbox{ тоді і лише тоді, коли } \quad ef=fe=e.
\end{equation*}
Так означений частковий порядок на $E(S)$ називається \emph{природним}.

Означимо відношення $\leqslant$ на інверсній напівгрупі $S$ так:
\begin{equation*}
    s\leqslant t \qquad \hbox{тоді і лише тоді, коли}\qquad s=te,
\end{equation*}
для деякого ідемпотента $e\in S$. Так означений частковий порядок називається \emph{природним част\-ковим порядком} на інверсній напівгрупі $S$~\cite{Lawson1998}. Очевидно, що звуження природного часткового порядку $\leqslant$ на інверсній напівгрупі $S$ на її в'язку $E(S)$ є природним частковим порядком на $E(S)$.

Нагадаємо (див., наприклад \cite[\S1.12]{CliffordPreston1961-1967}), що \emph{біциклічною напівгрупою} (або \emph{біцик\-лічним моноїдом}) ${\mathscr{C}}(p,q)$ називається напівгрупа з одиницею, породжена дво\-еле\-мент\-ною множиною $\{p,q\}$ і визначена одним визначальним співвідношенням $pq=1$. Біциклічна напівгрупа відіграє важливу роль у теорії
напівгруп. Зокрема, класична теорема Олафа Андерсена \cite{Andersen-1952} стверджує, що ($0$-)проста напівгрупа з ідемпотентом є цілком ($0$-)простою тоді і лише тоді, коли вона не містить ізоморфну копію біциклічної напівгрупи.

Через $\mathscr{I}_\lambda$ позначимо множину всіх часткових взаємнооднозначних перетворень кардинала $\lambda$ разом з такою напівгруповою операцією
\begin{equation*}
    x(\alpha\beta)=(x\alpha)\beta \quad \mbox{якщо} \quad
    x\in\operatorname{dom}(\alpha\beta)=\{
    y\in\operatorname{dom}\alpha\colon
    y\alpha\in\operatorname{dom}\beta\}, \qquad \mbox{для} \;
    \alpha,\beta\in\mathscr{I}_\lambda.
\end{equation*}
Напівгрупа $\mathscr{I}_\lambda$ називається  \emph{симетричною інверсною напівгрупою} або \emph{симетричним інверсним моноїдом} над кардиналом $\lambda$~(див.  \cite{CliffordPreston1961-1967}). Симетрична інверсна напівгрупа введена В. В. Вагнером у працях~\cite{Wagner-1952, Wagner-1952a} і вона відіграє важливу роль у теорії напівгруп.

Якщо $A$~--- підмножина евклідового простору $\mathbb{R}^n$, то через $\operatorname{int}A$  позначатимемо внутрішність  множини $A$ в просторі $\mathbb{R}^n$. Ми позначатимемо одиничну сферу та замкнену кулю радіуса $r>0$  в $\mathbb{R}^n$ через $\mathbb{S}^{n-1}$ i $\mathbf{B}_r$, відповідно.

Для довільних двох точок $x,y\in\mathbb{R}^n$ через $[x,y]$ позначатимемо відрізок в $\mathbb{R}^n$, який з'єднує точки $x,y$, тобто $[x,y]=\left\{z\in \mathbb{R}^n\colon \overrightarrow{xz}=\alpha\cdot\overrightarrow{xy}, 0\leqslant\alpha\leqslant 1\right\}$.

Компактна опукла підмножина в $\mathbb{R}^n$ з непорожньою внутрішністю називається \emph{опуклим тілом}. Підмножина $L\subseteq \mathbb{R}^n$ називається \emph{зіркою} в початку $\mathbf{0}$, якщо для довільної точки $x\in L$ відрізок $[\mathbf{0},x]$ міститься в $L$. Якщо $L$ є компактною підмножиною, яка є зіркою в початку $\mathbf{0}$, то її радіальна функція $\rho_L$ визначається для всіх $u\in \mathbb{S}^{n-1}$ так, що промінь відкладений у початку $\mathbf{0}$ паралельно до $u$ перетинає $L$, за формулою
\begin{equation*}
  (u)\rho_L=\max\left\{ c\geqslant 0\colon cx\in L\right\}.
\end{equation*}
Зауважимо, що в \cite{Gardner-Volcic-1994}, у якій введено всі вище означенні поняття, не припускається, що початок $\mathbf{0}$ належить зірці $L$. Надалі, ще крім того, ми будемо вважати, що $\mathbf{0}\in\operatorname{int}L$ для довільної зірки $L\subseteq \mathbb{R}^n$, що еквівалентно умові $\rho_L(u)\neq 0$ для всіх $u\in \mathbb{S}^{n-1}$, а також припускатимемо, що радіальна функція $\rho_L\colon \mathbb{S}^{n-1}\to L$ є неперервною.

Нехай $L_1$ i $L_2$~--- зірки в початку $\mathbf{0}$. Тоді гомеоморфізм $\alpha\colon L_1\to L_2$ називається \emph{зірковим}, якщо $(\mathbf{0})\alpha=\mathbf{0}$ і $([\mathbf{0},x])\alpha=[\mathbf{0},(x)\alpha]$ для довільного відрізка $[\mathbf{0},x]\subseteq L_1$. Надалі вважатимемо, що всі зірки $L\subseteq \mathbb{R}^n$ є в початку $\mathbf{0}$ і під \emph{частковими зірковими гомеоморфізмами} простору $\mathbb{R}^n$ будемо розуміти гомеоморфізми між зірками в $\mathbb{R}^n$.

З означення зіркового гомеоморфізму випливає твердження~\ref{proposition-1.0}.

\begin{proposition}\label{proposition-1.0}
Композиція двох часткових зіркових гомеоморфізмів простору $\mathbb{R}^n$ і обернене часткове відображення до часткового зіркового гомеоморфізму є част\-ковими зірковими гомеоморфізмами простору $\mathbb{R}^n$.
\end{proposition}

\begin{proposition}\label{proposition-1.1}
Довільні дві зірки в $\mathbb{R}^n$ є зірково гомеоморфними.
\end{proposition}

\begin{proof}
Доведемо, що довільна зірка $L$ зірково гомеоморфна одиничній кулі $\mathbf{B}_1$ в $\mathbb{R}^n$. Означимо відображення $\alpha_L\colon L\to \mathbf{B}_1$ наступним чином.  Нехай $\rho_L\colon \mathbb{S}^{n-1}\to L$~--- радіальна функ\-ція зірки $L$. Для довільного $x\in \mathbf{B}_1$ покладемо  $(x)r$~--- точка на одиничній сфері $\mathbb{S}^{n-1}$ така, що $x\in [\mathbf{0},(x)r]$. Тоді з неперервності радіальної функції $\rho_L\colon \mathbb{S}^{n-1}\to L$ випливає, що відображення $\beta_L\colon\mathbf{B}_1\to L$, означене за формулою
\begin{equation*}
(x)\beta_L=
\left\{
  \begin{array}{ll}
    \mathbf{0}, & \hbox{якщо~} x=\mathbf{0}; \\
    x\cdot ((x)r)\rho_L, & \hbox{в іншому випадку},
  \end{array}
\right.
\end{equation*}
є бієктивним і неперервним, і оскільки $\mathbf{B}_1$~--- компактний підпростір в $\mathbb{R}^n$, то $\beta_L$ є зірковим гомеоморфіз\-мом. Покладемо $\alpha_L\colon L\to \mathbf{B}_1$~--- обернене відображення до $\beta_L$. Очевидно, що ві\-доб\-ра\-жен\-ня $\alpha_L$ є зірковим гомеоморфіз\-мом.
\end{proof}

З твердження 14.1.7 \cite{Moszynska-20055} і з міркувань викладених в \cite[\S1.4, с.~29--30]{Engelking1989} випливає твердження~\ref{proposition-1.2}.

\begin{proposition}\label{proposition-1.2}
Перетин двох зірок є зіркою в $\mathbb{R}^n$.
\end{proposition}

Через $\mathbf{PStH}_{\mathbb{R}^n}$ позначимо множину всіх часткових зіркових гомеоморфізмів простору $\mathbb{R}^n$ з операцією композицією відображень.

З тверджень \ref{proposition-1.1} і \ref{proposition-1.2} випливає твердження~\ref{proposition-1.3}.

\begin{proposition}\label{proposition-1.3}
$\mathbf{PStH}_{\mathbb{R}^n}$ --- інверсна піднапівгрупа симетричного інверсного моноїда $\mathscr{I}_{\mathfrak{c}}$.
\end{proposition}

%\section{Мотивація досліджень і коротка історична довідка}

Дослідження автоморфізмів і груп автоморфізмів многовидів малої розмір\-нос\-ті фор\-му\-ють широку область сучасної математики, яка дуже швидко розвивається та роз\-та\-шо\-ва\-на на стику топології, алгебри й теорії динамічних систем. Ця область охоплює вивчення груп гомеоморфізмів прямої та кола, теорію автоморфізмів поверхонь і теорію груп класів відображень.

Автоморфізмам і групам автоморфізмів многовидів розмірності 1 і 2 присвячено фундаментальні праці Клейна, Фрике, Пуанкаре, Гурвіца, Дена, Данжуа, Алек\-сандера, Нільсена, Артіна, Керек'ярто, А. А. Маркова.
Сучасні дослідження груп гомеоморфізмів прямої викладено в оглядах Бекларяна \cite{Beklaryan-2004, Beklaryan-2015} та її застосування в теорії динамічних систем у монографії \cite{Katok-Hasselblatt-1995}.

Основні результати теорії напівгруп перетворень отримані в період
50-70-х років минулого століття, викладені в оглядах Меггіла
\cite{Magill1975} та Глускіна, Шайна, Шнепермана та Ярокера
\cite{Gluskin1977a}. У цьому напрямі працювали такі відомі
математики: Гауі, Гельфанд, Глускін, Грін, Енгелькінг, Кліффорд,
Ляпін, Меггіл, Престон, Саббах, Серпінський, Сушкевич, Улам,  Шайн,
Шнеперман, Шутов, Ярокер. На думку Меггіла (див. \cite{Magill1975}) Теорія напівгруп
неперервних перетворень топологічних просторів бере свій початок з
праць Глускіна \cite{Gluskin1959, Gluskin1959a, Gluskin1959b,
Gluskin1959c, Gluskin1960, Gluskin1960a}. В основному ці праці
Глускіна присвячені описанню структури напівгрупи $S(I)$ неперервних
перетворень одиничного відрізка $I$, а також описанню піднапівгруп
напівгрупи $S(X)$ неперервних перетворень топологічного прос\-то\-ру
$X$. Напівгрупу $S(I)$ неперервних перетворень одиничного відрізка
також дос\-лід\-жу\-вав Шутов у працях \cite{Shutov1963a,
Shutov1963b}, де він описав мак\-си\-мальну власну конгруенцію на
$S(I)$.

Напівгрупу $S(I)$ також досліджували в \cite{Insaridze, Shneperman1962,
Shneperman1962a, Shneperman1963, Shneperman1965, Shneperman1966,
Cezus, Jarnik, Mioduszewski, Rosicky1974, Rosicky1974a}, зокрема в
 \cite{Gluskin1959, Insaridze}  описано конгруенц-прості
піднапівгрупи в $S(I)$. Шне\-пер\-ман \cite{Shneperman1965a} та Уарндоф
\cite{Warndof} довели, що одиничний відрізок визначається
напівгрупою неперервних перетворень. Інші класи топологічних
просторів, що визначаються своїми напівгрупами неперервних
перетворень, описав Уарндоф у \cite{Warndof} і Росіцкий у
\cite{Rosicky1974, Rosicky1974a}. Зокрема такими є: локально зв'язні
сепарабельні метричні континууми, локально евклідові гаусдорфові
простори, нульвимірні метричні простори, $CW$-комплекси та ін.
Також О'Рейлі в праці \cite{Reilly} довела, що кожен гаусдорфовий
простір $X$ визначається напівгрупою усіх компактних відношень на
$X$.

Зауважимо, що група гомеоморфізмів дійсної прямої ізоморфна групі го\-мео\-мор\-фіз\-мів одиничного відрізка (інтервалу). Таким чином виникає задача: \emph{описати струк\-ту\-ру напівгрупи част\-кових
гомеоморфізмів топологічного простору $X$}, а в частковому
ви\-пад\-ку одиничного відрізка, чи дійсної прямої. Однією з останніх робіт з цієї тематики є праця Чучмана \cite{Chuchman2011ADM}, в якій описано структуру напівгрупи замкнених зв'язних част\-ко\-вих гомеоморфізмів
одиничного відрізка з однією нерухомою точкою. Також у \cite{Gutik-Melnyk-2015} описана структура напівгрупи $\mathscr{P\!\!H}^+_{\!\!\operatorname{\textsf{cf}}}\!(\mathbb{R})$ усіх монотонних коскінченних часткових гомеоморфізмів звичайної дійсної прямої $\mathbb{R}$. Зокрема, в \cite{Gutik-Melnyk-2015} доведено, що фактор-напівгрупа $\mathscr{P\!\!H}^+_{\!\!\operatorname{\textsf{cf}}}\!(\mathbb{R})/\mathfrak{C}_{\textsf{mg}}$ за найменшою груповою конгруенцією $\mathfrak{C}_{\textsf{mg}}$ ізоморфна групі $\mathscr{H}^+\!(\mathbb{R})$ усіх гомеоморфізмів, що зберігають орієнтацію простору $\mathbb{R}$, а також, що напівгрупа $\mathscr{P\!\!H}^+_{\!\!\operatorname{\textsf{cf}}}\!(\mathbb{R})$ ізоморфна напівпрямому добутку $\mathscr{H}^+\!(\mathbb{R})\ltimes_\mathfrak{h}\mathscr{P}_{\!\infty}(\mathbb{R})$ вільної напівґратки з одиницею $(\mathscr{P}_{\!\infty}(\mathbb{R}),\cup)$ з групою $\mathscr{H}^+\!(\mathbb{R})$.

Нагадаємо \cite{Chuchman2011ADM}, що часткове перетворення  $\alpha\colon X\rightharpoonup X$ топологічного простору $X$ називається \emph{замкненим зв'язним част\-ко\-вим гомеоморфізмом}, якщо його звуження $\alpha\colon \operatorname{dom}\alpha\to \operatorname{ran}\alpha$ є гомеоморфізмом і $\operatorname{dom}\alpha$ та $\operatorname{ran}\alpha$ --- замкнені зв'язні підмножини в $X$.
Очевидно, що кожен замкнений зв'язний част\-ко\-вий гомеоморфізм одиничного відрізка, чи прямої, з однією нерухомою точкою можна розглядати як зірковий частковий гомеоморфізм цього простору. Тому природно виникає задача про можливість поширення результатів, отриманих у праці \cite{Chuchman2011ADM} на вищі виміри.

\medskip

У цій праці досліджується структура напівгрупи $\mathbf{PStH}_{\mathbb{R}^n}$ зiркових часткових гомеоморфiзмiв скiнченновимiрного евклiдового простору $\mathbb{R}^n$. Описана структура ідемпотентів напівгрупи $\mathbf{PStH}_{\mathbb{R}^n}$ і відношення Ґріна на $\mathbf{PStH}_{\mathbb{R}^n}$. Зокрема доведено, що $\mathbf{PStH}_{\mathbb{R}^n}$ -- біпроста інверсна напівгрупа, а також, що кожна неодинична конгруенція на $\mathbf{PStH}_{\mathbb{R}^n}$ є груповою.
%%%%%%%%%%%%%%%%%%%%%%%%%%%%%%%%%%%%%%%%%%%%%%%%%%%%%

%\section{Алгебраїчні властивості напівгрупи $\mathbf{PStH}_{\mathbb{R}^n}$}

\medskip

Надалі в статті через $\textbf{St}_0(\mathbb{R}^n)$ позначатимемо множину усіх зірок в початку $\textbf{0}$. %Підмножина $A\subseteq \mathbb{R}^n$ називається \emph{ко-зіркою}, якщо $\mathbb{R}^n\setminus A$ є зіркою в $\mathbb{R}^n$. Через $\textbf{coSt}_0(\mathbb{R}^n)$

З означення напівгрупи $\mathbf{PStH}_{\mathbb{R}^n}$ і твердження \ref{proposition-1.3} випливає твердження~\ref{proposition-3.1}.

\begin{proposition}\label{proposition-3.1}
\begin{itemize}
  \item[$(i)$] Елемент $\alpha$ є ідемпотентом напівгрупи $\mathbf{PStH}_{\mathbb{R}^n}$ тоді і лише тоді, коли $\alpha\colon S\to S$~-- тотожне відображення для деякої зірки $S\in\textbf{St}_0(\mathbb{R}^n)$.
  \item[$(ii)$] В'язка $E(\mathbf{PStH}_{\mathbb{R}^n})$ ізоморфна напівґратці $(\textbf{St}_0(\mathbb{R}^n),\cap)$.
  \item[$(iii)$] $\varepsilon\leqslant\iota$ в $E(\mathbf{PStH}_{\mathbb{R}^n})$ тоді і лише тоді, коли $\operatorname{dom}\varepsilon\subseteq\operatorname{dom}\iota$.
  \item[$(iv)$] $\alpha\leqslant\beta$ в $\mathbf{PStH}_{\mathbb{R}^n}$ тоді і лише тоді, коли $\beta|_{\operatorname{dom}\alpha}=\alpha$.
\end{itemize}
\end{proposition}

Зауважимо, що пункти $(i)$ і $(ii)$ твердження \ref{proposition-3.1} випливають з того факту, що в симетричному інверсному моноїді ідемпотентами є лише тотожні часткові перетворення (див. \cite[твердження 1.1.2]{Lawson1998}), а пункти $(iii)$ і $(iv)$  випливають з означення природного часткового порядку на $E(\mathbf{PStH}_{\mathbb{R}^n})$ і $\mathbf{PStH}_{\mathbb{R}^n}$, відповідно, та з леми 3 \cite{Lawson1998}.

Якщо $S$ --- напівгрупа, то відношення Ґріна $\mathscr{R}$, $\mathscr{L}$, $\mathscr{J}$, $\mathscr{D}$ і $\mathscr{H}$ на $S$ визначаються так (див. \cite{Green-1951} або \cite[\S2.1]{CliffordPreston1961-1967}):
\begin{align*}
    &\qquad a\mathscr{R}b \mbox{~тоді і лише тоді, коли~} aS^1=bS^1;\\
    &\qquad a\mathscr{L}b \mbox{~тоді і лише тоді, коли~} S^1a=S^1b;\\
    &\qquad a\mathscr{J}b \mbox{~тоді і лише тоді, коли~} S^1aS^1=S^1bS^1;\\
    &\qquad \mathscr{D}=\mathscr{L}\circ\mathscr{R}= \mathscr{R}\circ\mathscr{L};\\
    &\qquad \mathscr{H}=\mathscr{L}\cap\mathscr{R}.
\end{align*}
Напівгрупа $S$ називається \emph{простою}, якщо $S$ не має власного двобічного ідеалу, тобто $S$ має єдиний $\mathscr{J}$-клас, і \emph{біпростою}, якщо $S$ має єдиний $\mathscr{D}$-клас.

Позаяк за твердженням \ref{proposition-1.3}, $\mathbf{PStH}_{\mathbb{R}^n}$ --- інверсна піднапівгрупа симетричного інверсного моноїда $\mathscr{I}_{\mathfrak{c}}$, то з визначень відношень Ґріна на $\mathscr{I}_{\mathfrak{c}}$ і твердження 3.2.11 \cite{Lawson1998} випливає твердження~\ref{proposition-3.2}.

\begin{proposition}\label{proposition-3.2}
Нехай $\alpha,\beta$ -- елементи напівгрупи $\mathbf{PStH}_{\mathbb{R}^n}$. Тоді:
\begin{itemize}
  \item[$(i)$] $\alpha\mathscr{R}\beta$ в $\mathbf{PStH}_{\mathbb{R}^n}$ тоді і лише тоді, коли $\operatorname{ran}\alpha=\operatorname{ran}\beta$;
  \item[$(ii)$] $\alpha\mathscr{L}\beta$ в $\mathbf{PStH}_{\mathbb{R}^n}$ тоді і лише тоді, коли $\operatorname{dom}\alpha=\operatorname{dom}\beta$;
  \item[$(iii)$] $\alpha\mathscr{H}\beta$ в $\mathbf{PStH}_{\mathbb{R}^n}$ тоді і лише тоді, коли $\operatorname{ran}\alpha=\operatorname{ran}\beta$ i $\operatorname{dom}\alpha=\operatorname{dom}\beta$.
\end{itemize}
\end{proposition}

\begin{proposition}\label{proposition-3.3}
$\mathbf{PStH}_{\mathbb{R}^n}$ -- біпроста напівгрупа.
\end{proposition}

\begin{proof}
Нехай $\varepsilon$ i $\iota$~-- довільні ідемпотенти напівгрупи $\mathbf{PStH}_{\mathbb{R}^n}$. Тоді за твердженням \ref{proposition-3.1}$(i)$ існують зірки  $E,I\in\textbf{St}_0(\mathbb{R}^n)$ такі, що $\varepsilon\colon \operatorname{dom}\varepsilon=E\to \operatorname{ran}\varepsilon=E$ i $\iota\colon \operatorname{dom}\iota=I\to \operatorname{ran}\iota=I$ -- тотожні відображення. За твердженням \ref{proposition-1.1} існує зірковий гомеоморфізм $\alpha\colon E\to I$. Очевидно, що $\alpha\alpha^{-1}=\varepsilon$ i $\alpha^{-1}\alpha=\iota$. Позаяк напівгрупа $\mathbf{PStH}_{\mathbb{R}^n}$ є інверсною, то з леми Манна (див. \cite[лема 1.1]{Munn1966}) випливає, що $\mathbf{PStH}_{\mathbb{R}^n}$ є біпростою напівгрупою.
\end{proof}

Оскільки кожен $\mathscr{D}$-клас елемента $a$ напівгрупи $S$ міститься в його $\mathscr{J}$-класі (див. \cite[\S2.1]{CliffordPreston1961-1967}), то з твердження~\ref{proposition-3.3} випливає наслідок~\ref{corollary-3.4}.

\begin{corollary}\label{corollary-3.4}
$\mathbf{PStH}_{\mathbb{R}^n}$ -- проста напівгрупа.
\end{corollary}

З твердження~\ref{proposition-3.3} і теореми 2.20 \cite{CliffordPreston1961-1967} випливає наслідок~\ref{corollary-3.5}.

\begin{corollary}\label{corollary-3.5}
Довільні дві максимальні підгрупи в $\mathbf{PStH}_{\mathbb{R}^n}$ є ізоморфними. Більше того кожна максимальна підгрупа в $\mathbf{PStH}_{\mathbb{R}^n}$ ізоморфна групі всіх зіркових гомеоморфізмів одиничної кулі $\mathbf{B}_1$ в $\mathbb{R}^n$.
\end{corollary}

\begin{lemma}\label{lemma-3.6}
Нехай $\mathfrak{C}$~--- конгруенція на напівгрупі $\mathbf{PStH}_{\mathbb{R}^n}$, $r_1$, $r_2$~--- довільні різні дійсні до\-дат\-ні числа та $\varepsilon_{r_1}\colon \mathbf{B}_{r_1}\to \mathbf{B}_{r_1}$ i $\varepsilon_{r_2}\colon \mathbf{B}_{r_2}\to \mathbf{B}_{r_2}$~--- тотожні відображення. Якщо $\varepsilon_{r_1}\mathfrak{C}\varepsilon_{r_2}$, то всі ідемпотенти напівгрупи $\mathbf{PStH}_{\mathbb{R}^n}$ є $\mathfrak{C}$-еквівалентними.
\end{lemma}

\begin{proof}
Спочатку доведемо, що для довільного  дійсного додатнього числа $r$ ідемпотент $\varepsilon_r$, де $\varepsilon_r\colon \mathbf{B}_r\to \mathbf{B}_r$~--- тотожне відображення, є $\mathfrak{C}$-еквівалентним ідемпотентам $\varepsilon_{r_1}$ i $\varepsilon_{r_2}$.

Не зменшуючи загальності, можемо вважати, що $r_1<r_2$. Розглянемо можливі випадки:
\begin{equation*}
  a)~r<r_1; \qquad b)~r_1<r<r_2; \qquad \hbox{i} \qquad c)~r_2<r.
\end{equation*}

%\textcolor[rgb]{1.00,0.00,0.00}{+++++++++++++++++++++++++}

Припустимо, що виконується випадок $b)~r_1<r<r_2$. Тоді з твердження \ref{proposition-3.1} випливає, що $\varepsilon_{r_1}=\varepsilon_{r_1}\varepsilon_r\mathfrak{C}\varepsilon_{r_2}\varepsilon_r=\varepsilon_r$, а отже, $\varepsilon_{r_1}\mathfrak{C}\varepsilon_r$ i $\varepsilon_r\mathfrak{C}\varepsilon_{r_2}$.

Припустимо, що виконується випадок $a)~r<r_1$. Означимо часткове відобра\-жен\-ня $\alpha\colon \mathbb{R}^n\rightharpoonup \mathbb{R}^n$ наступним чином:
\begin{equation*}
  \operatorname{dom}\alpha=\mathbf{B}_{r_2}, \qquad \operatorname{ran}\alpha=\mathbf{B}_{r_1} \qquad \hbox{i} \qquad (x)\alpha=x\cdot\frac{r_1}{r_2}.
\end{equation*}
Тоді часткове відображення $\alpha$ та обернене до нього $\alpha^{-1}$ є елементами напівгрупи $\mathbf{PStH}_{\mathbb{R}^n}$ й існує натуральне число $n_r$ таке, що $\left(\dfrac{r_1}{r_2}\right)^{n_r}<r$.

Нехай $S=\left\langle\alpha,\alpha^{-1}\right\rangle$~--- піднапівгрупа в $\mathbf{PStH}_{\mathbb{R}^n}$, породжена елементами $\alpha$ i $\alpha^{-1}$. Тоді легко бачити, що
\begin{equation*}
\varepsilon_{r_2}\alpha=\alpha\varepsilon_{r_2}=\alpha, \qquad \varepsilon_{r_2}\alpha^{-1}=\alpha^{-1}\varepsilon_{r_2}=\alpha^{-1}, \qquad \alpha\alpha^{-1}=\varepsilon_{r_2} \quad \hbox{i} \quad \alpha^{-1}\alpha=\varepsilon_{r_1}\neq \varepsilon_{r_2},
\end{equation*}
а отже, за лемою 1.31 з \cite{CliffordPreston1961-1967} напівгрупа $S=\left\langle\alpha,\alpha^{-1}\right\rangle$ ізоморфна біциклічному моноїдові ${\mathscr{C}}(p,q)$, причому всі елементи напівгрупи $S$ єдиним чином зоображаються у вигляді $\left(\alpha^{-1}\right)^i\alpha^j$, де $i$ та $j$ --- невід'ємні цілі числа, а ізоморфізм $\mathfrak{I}\colon S\to {\mathscr{C}}(p,q)$ визначається за формулою $\big(\left(\alpha^{-1}\right)^i\alpha^j\big)\mathfrak{I}=q^ip^j$. За наслідком 1.32 з \cite{CliffordPreston1961-1967} кожна неодинична конгруенція на біциклічному моноїді ${\mathscr{C}}(p,q)$ є груповою, а отже з наших припущень випливає, що усі ідемпотенти напівгрупи $S$ є $\mathfrak{C}$-еквівалентними. Тоді для $r_0=\left(\dfrac{r_1}{r_2}\right)^{n_r}$ маємо, що ідемпотент $\varepsilon_{r_0}$, де $\varepsilon_{r_0}\colon \mathbf{B}_{r_0}\to \mathbf{B}_{r_0}$~--- тотожне відобра\-жен\-ня, є $\mathfrak{C}$-еквівалентним ідемпотентам $\varepsilon_{r_1}$ i $\varepsilon_{r_2}$. Але за побудовою, $\varepsilon_{r_0}\leqslant\varepsilon_{r}\leqslant\varepsilon_{r_1}$ в $E(\mathbf{PStH}_{\mathbb{R}^n})$, а отже, з випадку $b)$ випливає, що $\varepsilon_r\mathfrak{C}\varepsilon_{r_2}$.

Припустимо, що виконується випадок $c)~r_2<r$. Тоді існує натуральне число $n_r$ таке, що $r_2\cdot \left(\dfrac{r_2}{r_1}\right)^{n_r}>r$. Означимо часткове відображення $\beta\colon \mathbb{R}^n\rightharpoonup \mathbb{R}^n$ наступним чином:
\begin{equation*}
  \operatorname{dom}\beta=\mathbf{B}_{r_2\left(\tfrac{r_2}{r_1}\right)^{n_r}}, \qquad \operatorname{ran}\beta=\mathbf{B}_{r_2\left(\tfrac{r_2}{r_1}\right)^{n_r-1}} \qquad \hbox{i} \qquad (x)\beta=x\cdot\frac{r_1}{r_2}.
\end{equation*}
Тоді часткове відображення $\beta$ та обернене до нього $\beta^{-1}$ є елементами напівгрупи $\mathbf{PStH}_{\mathbb{R}^n}$. Нехай $\varepsilon_{1}\colon \mathbf{B}_{r_2\left(\tfrac{r_2}{r_1}\right)^{n_r}}\to \mathbf{B}_{r_2\left(\tfrac{r_2}{r_1}\right)^{n_r}}$ і $\varepsilon_{2}\colon \mathbf{B}_{r_2\left(\tfrac{r_2}{r_1}\right)^{n_r-1}}\to \mathbf{B}_{r_2\left(\tfrac{r_2}{r_1}\right)^{n_r-1}}$~--- то\-тож\-ні відображення. Тоді очевидно, що $\varepsilon_{1}$ і  $\varepsilon_{2}$ є різними ідемпотентами напівгрупи $\mathbf{PStH}_{\mathbb{R}^n}$, причому $\varepsilon_2\leqslant\varepsilon_1$. Також легко бачити, що виконуються такі рівності
\begin{equation*}
\varepsilon_1\beta=\beta\varepsilon_1=\beta, \qquad \varepsilon_1\beta^{-1}=\beta^{-1}\varepsilon_1=\beta^{-1}, \qquad \beta\beta^{-1}=\varepsilon_1 \qquad \hbox{i} \qquad \beta^{-1}\beta=\varepsilon_2\neq \varepsilon_1,
\end{equation*}
а отже, за лемою 1.31 з \cite{CliffordPreston1961-1967} піднапівгрупа $T=\left\langle\beta,\beta^{-1}\right\rangle$ в $\mathbf{PStH}_{\mathbb{R}^n}$, породжена елементами $\beta$ i $\beta^{-1}$, ізоморфна біциклічному моноїдові ${\mathscr{C}}(p,q)$, причому всі елементи напівгрупи $T$ єдиним чином зоображаються у вигляді $\left(\beta^{-1}\right)^i\beta^j$, де $i$ та $j$ --- невід'ємні цілі числа, а ізоморфізм $\mathfrak{I}\colon T\to {\mathscr{C}}(p,q)$ визначається за формулою $\big(\left(\beta^{-1}\right)^i\beta^j\big)\mathfrak{I}=q^ip^j$. Очевидно, що ідемпотенти $\left(\beta^{-1}\right)^{n_r}\beta^{n_r}$ і $\left(\beta^{-1}\right)^{n_r+1}\beta^{n_r+1}$ напівгрупи  $\mathbf{PStH}_{\mathbb{R}^n}$, як часткові відображення, є тотожними відображеннями куль $\mathbf{B}_{r_2}$ i $\mathbf{B}_{r_1}$, відповідно, а отже, $\left(\beta^{-1}\right)^{n_r}\beta^{n_r}=\varepsilon_{r_2}$ і $\left(\beta^{-1}\right)^{n_r+1}\beta^{n_r+1}=\varepsilon_{r_1}$. Отож, два різні ідемпотенти напівгрупи $T$ є $\mathfrak{C}$-еквівалентними. За наслідком 1.32 з \cite{CliffordPreston1961-1967} кожна не\-тотожня конгруенція на бі\-цик\-ліч\-ному моноїді ${\mathscr{C}}(p,q)$ є груповою, а отже, з наших припущень випливає, що усі ідемпотенти напівгрупи $S$ є $\mathfrak{C}$-еквівалентними. Але за побудовою, $\varepsilon_{r_2}\leqslant\varepsilon_{r}\leqslant\varepsilon_{1}$ в $E(\mathbf{PStH}_{\mathbb{R}^n})$, а отже, з випадку $b)$ випливає, що $\varepsilon_r\mathfrak{C}\varepsilon_{r_2}$.
\end{proof}

\begin{lemma}\label{lemma-3.7}
Нехай $\mathfrak{C}$~--- конгруенція на напівгрупі $\mathbf{PStH}_{\mathbb{R}^n}$ така, що два різні ідемпотенти напівгрупи $\mathbf{PStH}_{\mathbb{R}^n}$ є $\mathfrak{C}$-еквівалентними. Тоді всі ідемпотенти напівгрупи $\mathbf{PStH}_{\mathbb{R}^n}$ є $\mathfrak{C}$-еквіва\-лент\-ни\-ми.
\end{lemma}

\begin{proof}
Припустимо, що $\varepsilon$ i $\iota$~--- різні $\mathfrak{C}$-еквівалентні ідемпотенти напівгрупи $\mathbf{PStH}_{\mathbb{R}^n}$. Тоді з $\varepsilon\mathfrak{C}\iota$ випливає, що $\varepsilon\iota\mathfrak{C}\iota\iota=\iota$, а отже, не зменшуючи загальності, можемо вважати, що $\varepsilon\leqslant\iota$ в $E(\mathbf{PStH}_{\mathbb{R}^n})$, а тоді за твердженням \ref{proposition-3.1}$(iii)$, $\operatorname{dom}\varepsilon\subseteq\operatorname{dom}\iota$.

Нехай $\alpha_\iota\colon \operatorname{dom}\iota\to \mathbf{B}_1$ --- частковий зірковий гомеоморфізм зірки $\operatorname{dom}\iota$ на одиничну кулю $\mathbf{B}_1$ в початку $\mathbf{0}$. За твердженням \ref{proposition-1.0} звуження $\alpha_\iota|_{\operatorname{dom}\varepsilon}\colon \operatorname{dom}\varepsilon\to (\operatorname{dom}\varepsilon)\alpha_\iota$ є частковим зірковим гомеоморфізмом зірки $\operatorname{dom}\varepsilon$ на зірку $(\operatorname{dom}\varepsilon)\alpha_\iota$. Тоді $\alpha_\iota^{-1}\alpha_\iota\colon \mathbf{B}_1\to \mathbf{B}_1$ --- тотожне відображення одиничної кулі $\mathbf{B}_1$ в початку $\mathbf{0}$. Позаяк
$
  \alpha_\iota^{-1}\alpha_\iota=\alpha_\iota^{-1}\iota\alpha_\iota\mathfrak{C}\alpha_\iota^{-1}\varepsilon\alpha_\iota
$
i $\alpha_\iota^{-1}\alpha_\iota\neq\alpha_\iota^{-1}\varepsilon\alpha_\iota$, то ідемпотент $\varepsilon_1$, $\varepsilon_1\colon \mathbf{B}_1\to \mathbf{B}_1$ --- тотожне відображення одиничної кулі $\mathbf{B}_1$ в початку $\mathbf{0}$, є $\mathfrak{C}$-еквівалентним деякому ідемпотентові $\varepsilon_s$ такому, що $\operatorname{dom}\varepsilon_s$ --- власна підмножина в $\mathbf{B}_1$. Тоді існує елемент $u_0\in \mathbb{S}^{n-1}$ такий, що $(u_0)\rho_{\operatorname{dom}\varepsilon_s}<1$. Позаяк радіальна функція $\rho_{\operatorname{dom}\varepsilon_s}\colon \mathbb{S}^{n-1}\to \mathbb{R}$ є неперервною, то існує відкритий окіл $U(u_0)$ точки $u_0$ на сфері $\mathbb{S}^{n-1}$ такий, що $(u)\rho_{\operatorname{dom}\varepsilon_s}<1$ для всіх $u\in U(u_0)$.

Для довільної точки $x\in \mathbb{S}^{n-1}\setminus U(u_0)$ через $\alpha_x\colon \mathbf{B}_1\to \mathbf{B}_1$ позначимо ортогональне перетворення одиничної кулі $\mathbf{B}_1$, яке відображає точку $u_0\in \mathbb{S}^{n-1}$ в $x\in \mathbb{S}^{n-1}$. Очевидно, що $\alpha_x\in \mathbf{PStH}_{\mathbb{R}^n}$, оскільки відображення $\alpha_x\colon \mathbf{B}_1\to \mathbf{B}_1$ є гомеоморфізмом, як елемент ортогональної групи матриць $O(n, \mathbb{R})$ (див. \cite{Kostrykin-Manin-1986}). Позаяк підпростір $\mathbb{S}^{n-1}\setminus U(u_0)$ в $\mathbb{S}^{n-1}$ є компактним, то відкрите покриття $\left\{(U(u_0))\alpha_x\colon x\in \mathbb{S}^{n-1}\setminus U(u_0)\right\}$  простору $\mathbb{S}^{n-1}\setminus U(u_0)$ містить скінченне підпокриття $\left\{(U(u_0))\alpha_{x_1},\ldots, (U(u_0))\alpha_{x_k}\right\}$. Позаяк $\varepsilon_1\mathfrak{C}\varepsilon_s$ i $\alpha_{x_i}^{-1}\varepsilon_1\alpha_{x_i}=\varepsilon_1$ для всіх $i=1,\ldots,k$, то $\varepsilon_1\mathfrak{C}\alpha_{x_i}^{-1}\varepsilon_s\alpha_{x_i}$ для кожного $i=1,\ldots,k$, а отже,
\begin{equation*}
  \varepsilon_1\mathfrak{C}\alpha_{x_1}^{-1}\varepsilon_s\alpha_{x_1}\cdots \alpha_{x_k}^{-1}\varepsilon_s\alpha_{x_k}.
\end{equation*}
Очевидно, що елемент $\alpha_{x_i}^{-1}\varepsilon_s\alpha_{x_i}$ є ідемпотентом напівгрупи $\mathbf{PStH}_{\mathbb{R}^n}$ для кожного $i=1,\ldots,k$, а отже, $\phi=\varepsilon_s\alpha_{x_1}^{-1}\varepsilon_s\alpha_{x_1}\cdots \alpha_{x_k}^{-1}\varepsilon_s\alpha_{x_k}$ є також ідемпотентом в $\mathbf{PStH}_{\mathbb{R}^n}$, оскільки $\mathbf{PStH}_{\mathbb{R}^n}$ --- інверсна напівгрупа, а ідемпотенти в інверсній напівгрупі комутують (див. \cite[теорема 1.17]{CliffordPreston1961-1967}). За побудовою, $(x)\rho_{\operatorname{dom}\phi}<1$ для довільного $x\in \mathbb{S}^{n-1}$, а оскільки радіальна функція $\rho_{\operatorname{dom}\phi}\colon \mathbb{S}^{n-1}\to \mathbb{R}$ є неперервною, то з компактності одиничної сфери $\mathbb{S}^{n-1}$ випливає, що відображення $\rho_{\operatorname{dom}\phi}$ на $\mathbb{S}^{n-1}$ набуває свого найбільшого значення. Нехай $R_{\phi}=\max\left\{(x)\rho_{\operatorname{dom}\phi}\colon x\in \mathbb{S}^{n-1}\right\}$ i $\varepsilon_{R_{\phi}}\colon \mathbf{B}_{\phi}\to \mathbf{B}_{\phi}$ --- то\-тож\-не відображення кулі радіуса $R_{\phi}$ в початку $\mathbf{0}$. Легко бачити, що $\varepsilon_{R_{\phi}}\in E(\mathbf{PStH}_{\mathbb{R}^n})$, $\varepsilon_{R_{\phi}}\phi=\phi$ i $\varepsilon_{R_{\phi}}\varepsilon_1=\varepsilon_{R_{\phi}}$. Тоді з умови $\varepsilon_1\mathfrak{C}\phi$ випливає, що $\varepsilon_{R_{\phi}}=\varepsilon_{R_{\phi}}\varepsilon_1\mathfrak{C}\varepsilon_{R_{\phi}}\phi=\phi$, а отже $\varepsilon_1\mathfrak{C}\varepsilon_{R_{\phi}}$. Далі скористаємося лемою \ref{lemma-3.6}.
\end{proof}

\begin{theorem}\label{theorem-3.8}
Кожна неодинична конгруенція на напівгрупі $\mathbf{PStH}_{\mathbb{R}^n}$ є груповою.
\end{theorem}

\begin{proof}
Нехай $\mathfrak{C}$~--- неодинична конгруенція на напівгрупі $\mathbf{PStH}_{\mathbb{R}^n}$. Тоді існують два різні $\mathfrak{C}$-еквівалентні елементи $\alpha$ i $\beta$ напівгрупи $\mathbf{PStH}_{\mathbb{R}^n}$.

Розглянемо можливі випадки:
\begin{enumerate}
  \item[(1)] елементи $\alpha$ i $\beta$ не належать одному $\mathscr{H}$-класу;
  \item[(2)] $\alpha\mathscr{H}\beta$.
\end{enumerate}

Припустимо, що виконується випадок (1). За твердженням 2.3.4 з \cite{Lawson1998}, $\alpha\alpha^{-1}\mathfrak{C}\beta\beta^{-1}$ i $\alpha^{-1}\alpha\mathfrak{C}\beta^{-1}\beta$.  З твердження \ref{proposition-3.2} випливає, що $\operatorname{ran}\alpha\neq\operatorname{ran}\beta$ або $\operatorname{dom}\alpha\neq\operatorname{dom}\beta$, а отже, за твердженням \ref{proposition-3.1}$(i)$ існують два різні $\mathfrak{C}$-еквівалентні ідемпотенти напівгрупи $\mathbf{PStH}_{\mathbb{R}^n}$. Далі скористаємося лемою \ref{lemma-3.7} і лемою I.7.10 з \cite{Petrich1984}.

Припустимо, що $\alpha\mathscr{H}\beta$. За теоремою 2.20 з \cite{CliffordPreston1961-1967} і наслідком \ref{corollary-3.5}, не зменшуючи загальності, можемо вважати, що $\alpha$ i $\beta$ --- зіркові гомеоморфізми одиничної кулі $\mathbf{B}_1$ в початку $\mathbf{0}$. Позаяк $\alpha\alpha^{-1}\mathfrak{C}\alpha\beta^{-1}$, то з вище сказаного випливає, що існує зірковий гомеоморфізм $\gamma\in \mathbf{PStH}_{\mathbb{R}^n}$ одиничної кулі $\mathbf{B}_1$, який є $\mathfrak{C}$-еквівалентним її тотожному відображенню $\varepsilon_1$ такий, що $\gamma\neq\varepsilon_1$. Тоді $(x)\gamma\neq x$ для деякого $x\in \mathbf{B}_1$.

Тоді виконується хоча б одна з умов:
\begin{enumerate}
  \item[$(a)$] існує елемент $x\in \mathbb{S}^{n-1}$ такий, що $([\mathbf{0},x])\gamma= [\mathbf{0},x]$ та існує $y\in[\mathbf{0},x]$ такий, що $(y)\gamma\neq y$;
  \item[$(b)$] існує елемент $x\in \mathbb{S}^{n-1}$ такий, що $([\mathbf{0},x])\gamma\neq [\mathbf{0},x]$.
\end{enumerate}

Припустимо, що виконується умова $(a)$. Припустимо, що $[\mathbf{0},y]\subsetneqq[\mathbf{0},(y)\gamma]$. У випадку $[\mathbf{0},y]\supsetneqq[\mathbf{0},(y)\gamma]$ міркування аналогічні.

Нехай $\mathbf{B}_y$~--- максимальна куля в початку $\mathbf{0}$, що містить точку $y$ і $\varepsilon_y\colon \mathbf{B}_y\to \mathbf{B}_y$~--- тотожне відображення кулі $\mathbf{B}_y$. Тоді $(y)\gamma\notin B_y$. Позаяк $\varepsilon_1\mathfrak{C}\gamma$, то $\varepsilon_y\varepsilon_1\mathfrak{C}\varepsilon_y\gamma$. Використавши твердження 2.3.4 з \cite{Lawson1998}, отримуємо, що
\begin{equation*}
\varepsilon_y=\varepsilon_y\varepsilon_1=
\varepsilon_1^{-1}\varepsilon_y^{-1}\varepsilon_y\varepsilon_1\mathfrak{C}\gamma^{-1}\varepsilon_y^{-1}\varepsilon_y\gamma=\gamma^{-1}\varepsilon_y\gamma.
\end{equation*}
Очевидно, що $\gamma^{-1}\varepsilon_y\gamma$~--- ідемпотент напівгрупи $\mathbf{PStH}_{\mathbb{R}^n}$, причому $(y)\gamma{\in} \operatorname{dom}(\gamma^{-1}\varepsilon_y\gamma)$, але $(y)\gamma\notin\operatorname{dom}\varepsilon_y$, а отже, за твердженням \ref{proposition-3.1}$(i)$ існують два різні $\mathfrak{C}$-еквівалентні ідемпотенти напівгрупи $\mathbf{PStH}_{\mathbb{R}^n}$. Далі скористаємося лемою \ref{lemma-3.7} і лемою I.7.10 з \cite{Petrich1984}.

Припустимо, що виконується умова $(b)$. Позаяк $\gamma$~--- частковий зірковий гомеоморфізм, то $(x)\gamma\neq x$. Існує відкрита $\delta$-куля $U_\delta((x)\gamma)$ точки $(x)\gamma$ в просторі $\mathbb{R}^n$, що не містить точку $x$. З метризовності простору $\mathbb{S}^{n-1}$ випливає, що він є цілком регулярним, а отже, існує неперервне відоб\-ра\-жен\-ня $f\colon \mathbb{S}^{n-1}\to [0,1]$ таке, що $((x)\gamma)f=1$ i $(z)f=0$ для всіх $z\in \mathbb{S}^{n-1}\setminus U_\delta((x)\gamma)$.

Визначимо зірку $L_\gamma$ в початку $\mathbf{0}$ наступним чином. Радіальною функцією зірки $L_\gamma$ є відображення $\rho_{L_\gamma}\colon \mathbb{S}^{n-1}\to \mathbb{R}^n$, яке визначається за формулою
 \begin{equation*}
 (z)\rho_{L_\gamma}=z\cdot (1+(z)f).
\end{equation*}
Очевидно, що так означене відображення $\rho_{L_\gamma}\colon \mathbb{S}^{n-1}\to \mathbb{R}^n$ є неперервним, а отже, відображення $\beta_{L_\gamma}\colon\mathbf{B}_1\to {L_\gamma}$, означене за формулою
\begin{equation*}
(z)\beta_{L_\gamma}=
\left\{
  \begin{array}{ll}
    \mathbf{0}, & \hbox{якщо~} z=\mathbf{0}; \\
    z\cdot ((z)r)\rho_{L_\gamma}, & \hbox{в іншому випадку},
  \end{array}
\right.
\end{equation*}
де $(z)r$~--- точка на одиничній сфері $\mathbb{S}^{n-1}$ така, що $z\in [\mathbf{0},(z)r]$, є бієктивним і неперервним, а оскільки $\mathbf{B}_1$~--- компактний підпростір в $\mathbb{R}^n$, то $\beta_{L_\gamma}$ є зірковим гомеоморфіз\-мом.

Позаяк $\varepsilon_1\mathfrak{C}\gamma$, то $\varepsilon_1\beta_{L_\gamma}\mathfrak{C}\gamma\beta_{L_\gamma}$.
За твердженням 2.3.4 з \cite{Lawson1998} маємо, що
\begin{equation*}
(\varepsilon_1\beta_{L_\gamma})^{-1}(\varepsilon_1\beta_{L_\gamma})\mathfrak{C}(\gamma\beta_{L_\gamma})^{-1}(\gamma\beta_{L_\gamma}),
\end{equation*}
а отже,
\begin{equation*}
\begin{split}
  \beta_{L_\gamma}^{-1}\beta_{L_\gamma} & = \beta_{L_\gamma}^{-1}\varepsilon_1\beta_{L_\gamma}= \\
    & =\beta_{L_\gamma}^{-1}\varepsilon_1^{-1}\varepsilon_1\beta_{L_\gamma}=\\
    & =(\varepsilon_1\beta_{L_\gamma})^{-1}(\varepsilon_1\beta_{L_\gamma})\mathfrak{C}(\gamma\beta_{L_\gamma})^{-1}(\gamma\beta_{L_\gamma})=\\
    & =\beta_{L_\gamma}^{-1}\gamma^{-1}\gamma\beta_{L_\gamma}.
\end{split}
\end{equation*}
Очевидно, що елементи $\beta_{L_\gamma}^{-1}\beta_{L_\gamma}$ i $\beta_{L_\gamma}^{-1}\gamma^{-1}\gamma\beta_{L_\gamma}$ є ідемпотентами напівгрупи $\mathbf{PStH}_{\mathbb{R}^n}$. Зауважимо, що  $(x)\rho_{L_\gamma}=x\cdot (1+(x)f)=x\cdot (1+0)=x$ і
\begin{equation*}
((x)\gamma)\rho_{L_\gamma}=(x)\gamma\cdot (1+((x)\gamma)f)=(x)\gamma\cdot (1+1)=(x)\gamma\cdot 2,
\end{equation*}
а отже, маємо, що $\beta_{L_\gamma}^{-1}\beta_{L_\gamma}\neq\beta_{L_\gamma}^{-1}\gamma^{-1}\gamma\beta_{L_\gamma}$. За твердженням \ref{proposition-3.1}$(i)$ існують два різні $\mathfrak{C}$-еквівалентні ідемпотенти напівгрупи $\mathbf{PStH}_{\mathbb{R}^n}$. Далі скористаємося лемою \ref{lemma-3.7} і лемою I.7.10 з \cite{Petrich1984}.
\end{proof}

\bigskip

Автори висловлюють подяку рецензентові за суттєві зауваження та поради.
%%%%%%%%%%%%%%%%%%%%%%%%%%%%%%%%%%%%%%%%%%%%%%%%%%%%%%%%%%%%
%\vskip 5pt

\vskip1cm

\medskip



 \begin{thebibliography}{11}
 \bibitem{Beklaryan-2004}
Л. А. Бекларян,
\emph{Группы гомеоморфизмов прямой и окружности. Топологические ха\-рак\-теристики и метри\-чес\-кие инварианты},
УМН  \textbf{59}  (2004), no. 4(358), 3--68. 
DOI: 10.4213/rm758;
\textbf{English version}:
L. A. Beklaryan,
\emph{Groups of homeomorphisms of the line and the circle. Topological characteristics and metric invariants},
Russian Math. Sur\-veys \textbf{59} (2004), no. 4,  599--660.	
DOI: 10.1070/RM2004v059n04ABEH000758
%
\bibitem{Beklaryan-2015}
Л. А. Бекларян,
\emph{Группы гомеоморфизмов прямой и окружности. Метрические инварианты и вопросы классификации},
УМН, \textbf{70} (2015), no. 2(422), 3--54.  
DOI: 10.4213/rm9654;
\textbf{English version}:
L.~A.~Beklaryan,
\emph{Groups of line and circle ho\-me\-o\-mor\-phisms. Metric invariants and questions of classification},
Russian Math. Surveys  \textbf{70} (2015), no. 2, 203--248.
DOI: 10.1070/RM2015v070n02ABEH004946
%
\bibitem{Wagner-1952}
В. В. Вагнер,
\emph{К теорнн частичных преобразований},
ДАН СССР  \textbf{84} (1952), 653--656.
%
\bibitem{Wagner-1952a}
В. В. Вагнер,
\emph{Обобшенные группы},
ДАН СССР  \textbf{84} (1952), 1119–1122.
%
\bibitem{Gluskin1959}
Л.~M.~Глускин, \emph{{Полугруппа гомеоморфных отображений отрезка}},
Матем. сб. \textbf{49} (1959), no. 1(91), 13--28.
%
\bibitem{Gluskin1959a}
Л.~M.~Глускин,
\emph{Полугруппы топологических отображений},
ДАН СССР \textbf{125} (1959), 699--702.
%
\bibitem{Gluskin1959b}
Л.~M.~Глускин,
\emph{Транзитивные полугруппы преобразований},
ДАН СССР \textbf{129} (1959), 16--18.
%
\bibitem{Gluskin1959c}
Л.~M.~Глускин,
\emph{Идеалы полугрупп преобразований},
Матем. сб. \textbf{47} (1959), no. 1(89), 111--130.
%
\bibitem{Gluskin1960}
Л.~М.~Глускін,
\emph{Про одну півгруппу непреривиих функцій},
Доповіді АН УРСР \textbf{5} (1960), 582--585.
%
\bibitem{Gluskin1960a}
Л.~М.~Глускин,
\emph{Полугруппы топологических преобразований},
Изв. вузов. Матем. (1963), no. 1, 54--65.
%
\bibitem{Gutik-Melnyk-2015}
О. Гутік, К. Мельник,
\emph{Напiвгрупа монотонних ко-скiнченних часткових гомеомор\-фiз\-мiв дiйсної прямої},
Мат. вісник Наук. тов. ім. T. Шевченка \textbf{12} (2015), 24--40.
%
\bibitem{Insaridze}
Х. Н. Инасаридзе,
 \emph{О простых полугруппах},
Матем. сб. \textbf{57} (1962), no. 2(99), 225--232.
%
\bibitem{Kostrykin-Manin-1986}
А. И. Кострикин, Ю. И. Манин,
\emph{Линейная алгебра и геометрия}.
Учебное пособие для вузов.  2-е изд., перераб.  Наука, Москва, 1986.
%
\bibitem{Shneperman1962}
Л.~Б.~Шнеперман,
\emph{Полугруппы непрерывных преобразований},
ДАН СССР \textbf{144} (1962), no. 3, 509--511.
%
\bibitem{Shneperman1962a}
Л.~Б.~Шнеперман,
\emph{Полугруппы непрерывных преобразований и гомеоморфизмов прос\-той дуги},
ДАН СССР \textbf{146} (1962), 1301--1304.
%
\bibitem{Shneperman1963}
Л.~Б.~Шнеперман,
\emph{Полугруппы непрерывных преобразований метрических прост\-ранств},
Матем. сб. \textbf{61} (1963), no. 3(103), 306--318.
%
\bibitem{Shneperman1965}
Л.~Б.~Шнеперман,
\emph{Полугруппы непрерывных преобразований замкнутых
множеств числовой прямой},
Изв. вузов. Матем.  (1965), no. 6,  166--175.
%
\bibitem{Shneperman1965a}
Л.~Б.~Шнеперман,
\emph{Полугруппы непрерывных преобразований топологических прост\-ранств},
Сиб. матем. журн. \textbf{4} (1965), no. 1, 221--229.
%
\bibitem{Shneperman1966}
Л.~Б.~Шнеперман,
\emph{Полугруппа гомеоморфизмов простой дуги},
Изв. вузов. Матем.  (1966), no. 2,  127--136.
%
\bibitem{Shutov1963a}
Э. Г. Шутов,
\emph{О гомоморфизмах некоторых полугрупп непрерывнвх функций},
Сиб. матем. журн. \textbf{4} (1963), no. 3, 695--701.
%
\bibitem{Shutov1963b}
Э. Г. Шутов,
\emph{О гомоморфизмах некоторых полугрупп непрерывных монотонных функ\-ций},
Сиб. матем. журн. \textbf{4}:4 (1963), 944--950.
%
\bibitem{Andersen-1952}
O. Andersen,
\emph{Ein Bericht uber die Struk\-tur abstrakter Halbgruppen},
PhD Thesis. Ham\-burg, 1952.
%
\bibitem{Anderson1958}
R. D. Anderson,
\emph{The algebraic simplicity of certain groups of homeomorphisms},
Amer. J. Math. \textbf{80} (1958), no. 4, 955---963.
DOI: 10.2307/2372842
%
\bibitem{Cezus}
F.~A.~Cezus,
\emph{Green's  relations  in  semigroups of  functions},
 Ph.D. Thesis, Australian  National University, Canberra,
Australia, 1972.
%
\bibitem{Chuchman2011ADM}
I.~Chuchman,
\emph{On a semigroup of closed connected partial homeomorphisms of the unit in\-ter\-val with a fixed point},
Algebra Discr. Math. \textbf{12} (2011), no. 2, 38--52.
%
\bibitem{CliffordPreston1961-1967}
A.~H.~Clifford and G.~B.~Preston,
\emph{The Algebraic theory of semigroups}, Vols. I and II,
Amer. Math. Soc. Surveys {\bf 7}, Providence, R.I.,  1961 and  1967.
%
\bibitem{Engelking1989}
R.~Engelking,
\emph{General topology}, 2nd ed., Heldermann, Berlin, 1989.
%
\bibitem{Gardner-Volcic-1994}
R. J. Gardner and A. Vol\v{c}i\v{c},
\emph{Tomography of convex and star bodies},
Adv. Math. \textbf{108} (1994), no. 2, 367--399.
DOI: 10.1006/aima.1994.1075
%
\bibitem{Gluskin1977a}
L.~M.~Gluskin, B.~M.~Schein, L.~B. \v{S}neperman, and
I.~S.~Yaroker,
\emph{Addendum to a sur\-vey of semigroups of continuous
self-maps},
Semigroup Forum \textbf{14} (1977), no. 1, 95--125. 
DOI: 10.1007/BF02194658
%
\bibitem{Green-1951}
J. A. Green,
\emph{On the structure of semigroups},
Ann. Math. Ser. 2 \textbf{54} (1951), no. 1,  163--172.
DOI: 10.2307/1969317
%
\bibitem{Jarnik}
 V. Jarn\'{\i}k  et V. Knichal,
\emph{Sur  l'approximation  des  fonctions continues  par  les  superpositlons  de deux  fonctlons},
Fund.  Math. \textbf{24} (1935), no. 1, 206--208.
DOI: 10.4064/fm-24-1-206-208
%
\bibitem{Katok-Hasselblatt-1995}
A. Katok and B. Hasselblatt,
\emph{Introduction to the modern theory of dynamical systems},
Encyclopedia of Mathematics and its Applications, vol. 54, Cambridge University Press, Cambridge, 1995.
%
\bibitem{Lawson1998}
M.~V. Lawson,
\emph{Inverse semigroups. The theory of partial symmetries},
World Scientific, Sin\-ga\-pore, 1998.

\bibitem{Magill1975}
K.~D.~Magill, jr.,
\emph{A survey of semigroups of continuous selfmaps},
Semigroup Forum \textbf{11} (1975/1976), no. 1,  189--282.
DOI: 10.1007/BF02195270
%
\bibitem{Mioduszewski}
J. V. Mioduszewski,
\emph{On a  quasi-ordering  in the class  of
continuous mappings of a  closed  interval  into itself},
Colloq. Math. \textbf{9} (1962), no. 2, 233--240.
DOI: 10.4064/cm-9-2-233-240
%
\bibitem{Moszynska-20055}
M. Moszy\'{n}ska,
\emph{Selected topics in convex geometry},
Birkh\"{a}user, Basel, 2005.
%
\bibitem{Munn1966}
W.~D.~Munn,
\emph{Uniform semilattices and bisimple inverse semigroups},
Quart. J. Math.  \textbf{17}  (1966), no. 1, 151--159.
DOI: 10.1093/qmath/17.1.151
%
\bibitem{Reilly}
S.~B.~O'Reilly,
\emph{The characteristic semigroup of topological space},
General Topology Appl. \textbf{5} (1975), no. 2,  95--106.
DOI: 10.1016/0016-660X(75)90015-X
%
\bibitem{Petrich1984}
M.~Petrich,
\emph{Inverse semigroups},
John Wiley $\&$ Sons, New York, 1984.
%
\bibitem{Rosicky1974}
J. V. Rosick\'{y},
\emph{Remarks  on  topologies  uniquely determined  by
their  continuous  self maps},
Czech. Math. J. \textbf{24}(\textbf{99}) (1974), no. 3, 373--377.
%
\bibitem{Rosicky1974a}
J. V. Rosick\'{y},
\emph{The topology of  the  unit  interval is  not uniquely
determined  by its  continuous self maps  among  set  systems},
Colloq.  Math. \textbf{31} (1974), no. 2, 179--188. 
DOI: 10.4064/cm-31-2-179-188
%
\bibitem{Warndof}
J. C. Warndof,
\emph{Topologies  uniquely  determined by their continuous selfmaps},
Fund. Math. \textbf{66}  (1970), no. 1,  25--43.
DOI: 10.4064/fm-66-1-25-43

 \end{thebibliography}
\end{document}